\documentclass[11pt]{article}

\usepackage{amssymb}
\usepackage{amsmath}
\usepackage{color}
\usepackage{theorem}
\usepackage[all]{xy}
\usepackage[normalem]{ulem}

\newcommand{\XYMATRIX}{\xymatrix@M=6pt}
\newcommand{\aremb}{\ar@{^{(}->}}
\newcommand{\arembfrom}{\ar@{<-^{)}}}

\numberwithin{equation}{section}

\theorembodyfont{\slshape}

  \newtheorem{THM}{Theorem}[section]
  \newtheorem{LEM}[THM]{Lemma}

\theorembodyfont{\rmfamily}

  \newtheorem{DEF}[THM]{Definition}
  \newtheorem{EX}{Example}[section]

\newif\ifQEDsign
\newcommand{\QED}{\global\QEDsigntrue\hfill$\square$}

\newenvironment{proof}%
    {\par\noindent\textit{Proof.}\global\QEDsignfalse}%
    {\ifQEDsign\else\QED\fi\par\bigskip\par}

\newcommand{\lex}{\mathrel{<_{\mathit{lex}}}}
\newcommand{\alex}{\mathrel{<_{\mathit{alex}}}}
\newcommand{\galex}{\mathrel{>_{\mathit{alex}}}}
\newcommand{\clex}{\mathrel{<_{\overline{\mathit{lex}}}}}

\renewcommand{\le}{\leqslant}
\renewcommand{\ge}{\geqslant}

\newcommand{\0}{\varnothing}

\renewcommand{\sec}{\cap}
\renewcommand{\phi}{\varphi}
\renewcommand{\epsilon}{\varepsilon}
\newcommand{\UNION}{\bigcup}

\newcommand{\CC}{\mathbf{C}}
\newcommand{\DD}{\mathbf{D}}

\newcommand{\KK}{\mathbf{K}}

\newcommand{\union}{\cup}
\newcommand{\restr}[2]{\hbox{$#1$}\hbox{$\upharpoonright$}_{#2}}
\newcommand{\reduct}[2]{\hbox{$#1$}\hbox{$|$}_{#2}}
\newcommand{\Boxed}[1]{\mbox{$#1$}}

\newcommand{\Ob}{\mathrm{Ob}}

\newcommand{\arity}{\mathrm{ar}}

\newcommand{\op}{\mathrm{op}}

\newcommand{\calA}{\mathcal{A}}
\newcommand{\calB}{\mathcal{B}}
\newcommand{\calC}{\mathcal{C}}
\newcommand{\calD}{\mathcal{D}}
\newcommand{\calE}{\mathcal{E}}

\newcommand{\calL}{\mathcal{L}}

\newcommand{\calP}{\mathcal{P}}

\newcommand{\calS}{\mathcal{S}}

\newcommand{\calX}{\mathcal{X}}

\newcommand{\catH}{{\mathbf{H}}}

\newcommand{\catGR}{\mathbf{GR}}
\newcommand{\catREL}{\mathbf{Rel}}
\newcommand{\catAREL}{\overrightarrow{\mathbf{Rel}}}

\DeclareMathOperator{\tp}{tp}

\DeclareMathOperator{\mat}{mat}
\DeclareMathOperator{\tup}{tup}

\title{A New Proof of the Ne\v set\v ril-R\"odl Theorem\\\textcolor{blue}{(CORRECTED VERSION)}}
\author{%
  Dragan Ma\v sulovi\'c\\
  University of Novi Sad, Faculty of Sciences\\
  Department of Mathematics and Informatics\\
  Trg Dositeja Obradovi\'ca 3, 21000 Novi Sad, Serbia\\
  e-mail: dragan.masulovic@dmi.uns.ac.rs}

\begin{document}
\maketitle

\begin{abstract}
  In this paper we give a new proof of the Ne\v set\v ril-R\"odl Theorem, a deep result of discrete mathematics
  which is one of the cornerstones of the structural Ramsey theory. In contrast to the well-known proofs
  which employ intricate combinatorial strategies, this proof is spelled out in the language of
  category theory and the main result follows by applying several simple categorical constructions.
  The gain from the approach we present here is that, instead of giving the proof in the form of a large
  combinatorial construction, we can start from a few building blocks and then combine them into the final proof using general principles.

  \bigskip


  \noindent \textbf{AMS Subj.\ Classification (2010):} 18A10, 05C55
\end{abstract}

\section{Introduction}

Generalizing the classical results of F.~P.~Ramsey from the late 1920's, the structural Ramsey theory originated at
the beginning of 1970’s in a series of papers (see \cite{N1995} for references).
We say that a class $\KK$ of finite structures has the \emph{Ramsey property} if the following holds:
for any number $k \ge 2$ of colors and all $\calA, \calB \in \KK$ such that $\calA$ embeds into $\calB$
there is a $\calC \in \KK$
such that no matter how we color the copies of $\calA$ in $\calC$ with $k$ colors, there is a \emph{monochromatic} copy
$\calB'$ of $\calB$ in $\calC$ (that is, all the copies of $\calA$ that fall within $\calB'$ are colored by the same color).

One of the cornerstones of the structural Ramsey theory is the Ne\v set\v ril-R\"odl Theorem which states that the class
of all finite linearly ordered relational structures (all having the same, fixed, relational type) has the Ramsey property
\cite{AH}, \cite{Nesetril-Rodl,Nesetril-Rodl-1983}. The fact that this result has been proved independently by several research teams, and then reproved
in various ways and in various contexts \cite{AH,Nesetril-Rodl-1983,Nesetril-Rodl-1989,Promel-1985}
clearly demonstrates the importance and justifies the distinguished status
this result has in discrete mathematics.

In this paper we give yet another proof of the Ne\v set\v ril-R\"odl Theorem. In contrast to the well-known proofs
which employ intricate combinatorial strategies, our proof starts from (a categorical version of)
the Graham-Rothschild Theorem~\cite{GR} and then transfers the Ramsey property
from the Graham-Rothschild category (see Example~\ref{opos.ex.GR} for the definition)
to the category of finite linearly ordered relational structures
using products of categories, pre-adjunctions (see Section~\ref{nrt.sec.transfer}
for the definition) and passing to a special subcategory.
The gain from the approach we present here is that,
instead of giving the proof in the form of a large combinatorial construction, we can start from a few building blocks and
then combine them into the final proof using general categorical principles.

In Section~\ref{nrt.sec.prelim} we give a brief overview of standard notions referring to first order structures
and formulate the Ramsey property in the language of category theory.
In Section~\ref{nrt.sec.transfer} we discuss the invariance of the Ramsey property under
finite products of categories, (a particular form of) pre-adjunctions, and under passing to
special subcategories. The corresponding results were proved in~\cite{masulovic-ramsey,masul-preadj,masul-drp-perm} but
in order to make the paper self-contained we provide brief sketches of the proofs.
Finally, in Section~\ref{nrt.sec.the-proof} we use the four results from Section~\ref{nrt.sec.transfer} to
present a new, categorical proof of the Ne\v set\v ril-R\"odl Theorem.
Let us underline that in this paper we do not consider the more general version of the theorem
which shows the Ramsey property also for classes of structures defined by forbidden substructures.

\section{Preliminaries}
\label{nrt.sec.prelim}

In order to fix notation and terminology in this section we give a brief overview of standard notions referring to first order structures
and formulate the Ramsey property in the language of category theory.
For a systematic treatment of category-theoretic notions we refer the reader to~\cite{AHS}.

\subsection{Structures}

Let $\Theta$ be a set of function and relation symbols.
A \emph{$\Theta$-structure} $\calA = (A, \Theta^\calA)$ is a set $A$ together with a set $\Theta^\calA$ of
functions and relations on $A$ which are interpretations of the corresponding symbols in $\Theta$.
The underlying set of a structure $\calA$, $\calA_1$, $\calA^*$, \ldots\ will always be denoted by its roman letter $A$, $A_1$, $A^*$, \ldots\ respectively.
A structure $\calA = (A, \Theta^\calA)$ is \emph{finite} if $A$ is a finite set.

An \emph{embedding} $f: \calA \hookrightarrow \calB$ is an injection $f: A \rightarrow B$ which respects
functions, and preserves and reflects the relations.
Surjective embeddings are \emph{isomorphisms}. We write $\calA \cong \calB$ to denote that $\calA$ and $\calB$
are isomorphic, and $\calA \hookrightarrow \calB$ to denote that there is an embedding of $\calA$ into $\calB$.

A structure $\calA$ is a \emph{substructure} of a structure
$\calB$ ($\calA \le \calB$) if the identity map is an embedding of $\calA$ into $\calB$.
Let $\calA$ be a structure and $\0 \ne B \subseteq A$. Then $\restr \calA B = (B, \restr{\Theta^\calA}{B})$ denotes
the \emph{substructure of $\calA$ induced by~$B$}, where $\restr{\Theta^\calA}{B}$ denotes the restriction of each
function and relation in $\Theta^\calA$ to~$B$.
Note that $\restr \calA B$ is not required to exist for every $B \subseteq A$. For example, if $\Theta^\calA$ contains functions,
only those $B$ which are closed with respect to all the functions in $\Theta^\calA$ qualify for the base set of a substructure.

If $\calA$ is a $\Theta$-structure and $\Sigma \subseteq \Theta$ then by $\reduct \calA \Sigma$ we denote the $\Sigma$-reduct of~$\calA$:
$\reduct \calA \Sigma = (A, \Sigma^\calA)$.

Let $\calL = (L, \Boxed<)$ be a finite linearly ordered set.
For a nonempty $X \subseteq L$ let $\min_\calL(X)$, resp.\ $\max_\calL(X)$, denote the minimum, resp.\ maximum, of $X$ in $\calL$.
As a convention we let $\min_\calL \0 = \mathstrut$the top element of $\calL$, and
$\max_\calL \0 = \mathstrut$the bottom element of $\calL$.

Let $\lex$, $\alex$ and $\clex$ denote the \emph{lexicographic}, \emph{anti-lexicographic} and
\emph{complemented lexicographic} ordering on $\calP(L)$, respectively, defined as follows:
\begin{align*}
A \lex B \text{ iff }  A \subset B, &\text{ or }\min\nolimits_\calL(B \setminus A) < \min\nolimits_\calL(A \setminus B) \text{ in case}\\
                                           &\text{ $A$ and $B$ are incomparable};\\
A \alex B \text{ iff } A \subset B, &\text{ or }\max\nolimits_\calL(A \setminus B) < \max\nolimits_\calL(B \setminus A) \text{ in case}\\
                                           &\text{ $A$ and $B$ are incomparable};\\
A \clex B \text{ iff } A \supset B, &\text{ or }\min\nolimits_\calL(A \setminus B) < \min\nolimits_\calL(B \setminus A) \text{ in case}\\
                                            &\text{ $A$ and $B$ are incomparable}.
\end{align*}
(Note that $A \clex B$ iff $L \setminus A \lex L \setminus B$, hence the name.)
It is easy to see that all these are linear orders on $\calP(L)$.

\subsection{The Ramsey property in the language of category theory}

Let $\CC$ be a category and $\calS$ a set. We say that
$
  \calS = \calX_1 \union \ldots \union \calX_k
$
is a \emph{$k$-coloring} of $\calS$ if $\calX_i \sec \calX_j = \0$ whenever $i \ne j$.
Equivalently, a $k$-coloring of $\calS$ is any map $\chi : \calS \to \{1, 2, \ldots, k\}$.
For an integer $k \ge 2$ and $A, B, C \in \Ob(\CC)$, the class of objects of $\CC$, we write
$
  C \longrightarrow (B)^{A}_k
$
to denote that for every $k$-coloring
$
  \hom_\CC(A, C) = \calX_1 \union \ldots \union \calX_k
$
there is an $i \in \{1, \ldots, k\}$ and a morphism $w \in \hom_\CC(B, C)$ such that
$w \cdot \hom_\CC(A, B) \subseteq \calX_i$.

\begin{DEF}
  A category $\CC$ has the \emph{Ramsey property} if
  for every integer $k \ge 2$ and all $A, B \in \Ob(\CC)$
  such that $\hom_\CC(A, B) \ne \0$ there is a
  $C \in \Ob(\CC)$ such that $C \longrightarrow (B)^{A}_k$.
\end{DEF}

\begin{EX}
  For $b \ge 2$ a \emph{linearly ordered $b$-uniform hypergraph} is a structure $\calA = (A, E, \Boxed<)$
  where $A$ is a nonempty set of \emph{vertices} of $\calA$, $E$ is a set of $b$-subsets of $A$ whose elements are
  called the \emph{hyperedges} of $\calA$ and $<$ is a linear order on~$A$.
  An \emph{embedding} between two linearly ordered $b$-uniform hypergraphs $\calA = (A, E, \Boxed<)$ and $\calB = (B, F, \Boxed<)$
  is an injective map $f : A \to B$ such that $e \in E$ if and only if $f(e) \in F$ for every $e \in E$.

  Let $\catH(b)$, $b \ge 2$, denote the category whose objects are finite linearly ordered $b$-uniform hypergraphs and whose
  morphisms are embeddings.
  The category $\catH(b)$ has the Ramsey property for every $b \ge 2$~\cite{AH,Nesetril-Rodl-1983}.
\end{EX}

\begin{EX}
  Let $\Theta$ be a relational language and let $\Boxed< \notin \Theta$ be a binary relational symbol.
  A \emph{linearly ordered $\Theta$-structure} is a $(\Theta\union\{\Boxed<\})$-structure
  $\calA = (A, \Theta^\calA, \Boxed{<^\calA})$ where $<^\calA$ is a linear order on~$A$.

  By $\catREL(\Theta, \Boxed<)$ we denote the category whose objects are finite linearly ordered $\Theta$-structures
  and whose morphisms are embeddings.
  The category $\catREL(\Theta, \Boxed<)$ has the Ramsey property. This is the famous
  Ne\v set\v ril-R\"odl Theorem~\cite{AH, Nesetril-Rodl}.
\end{EX}

\begin{EX}\label{opos.ex.GR}
  Let $X = \{x_1, x_2, \ldots\}$ be a countably infinite set of variables and let
  $A$ be a finite alphabet disjoint from $X$. An \emph{$m$-parameter word over $A$ of length $n$} is a word
  $w \in (A \union \{x_1, x_2, \ldots, x_m\})^n$ satisfying the following:
  \begin{itemize}
  \item
    each of the letters $x_1, \ldots, x_m$ appears at least once in $w$, and
  \item
    $\min(w^{-1}(x_i)) < \min(w^{-1}(x_j))$ whenever $1 \le i < j \le m$.
  \end{itemize}
  (Here $w^{-1}(a)$ denotes the set of all the positions in $w$ where $a$ appears.)

  Let $W^n_m(A)$ denote the set of all the $m$-parameter words over $A$ of length~$n$.
  For $u \in W^n_m(A)$ and $v = v_1 v_2 \ldots v_m \in W^m_k(A)$ let
  \begin{equation}\label{nrt.eq.dot}
    u \cdot v = u[v_1/x_1, v_2/x_2, \ldots, v_m/x_m] \in W^n_k(A)
  \end{equation}
  denote the word obtained by replacing each occurence of $x_i$ in $u$ with $v_i$,
  simultaneously for all $i \in \{1, \ldots, m\}$.
  
  Let $\catGR(A, X)$ denote the \emph{Graham-Rothschild category over $A$ and $X$} whose objects are positive integers 1, 2, \ldots,
  whose morphisms are given by $\hom(k, n) = W^n_k(A)$ if $k \le n$ and $\hom(k, n) = \0$ if $k > n$, and where the composition
  of morphisms is defined in~\eqref{nrt.eq.dot}.
  For every finite set $A$ and a countably infinite set $X = \{x_1, x_2, \ldots\}$ disjoint from $A$ the Graham-Rothschild
  category $\catGR(A, X)$ has the Ramsey property. This is the famous Graham-Rothschild Theorem~\cite{GR}.
\end{EX}

\section{Transferring the Ramsey property between categories}
\label{nrt.sec.transfer}

Our proof of the Ne\v set\v ril-R\"odl Theorem relies on the idea of transferring the Ramsey property from a category
which is known to posses it (such as the Graham-Rothschild category) to the category we are interested in using some
``transfer principles''. In this section we collect three such principles proved in~\cite{masul-preadj,masul-drp-perm}.
In order to make the paper self-contained we also provide sketches of proofs.

A pair of \emph{maps}
$
    F : \Ob(\DD) \rightleftarrows \Ob(\CC) : G
$
is a \emph{pre-adjunction between the categories $\CC$ and $\DD$} \cite{masul-preadj} provided there is a family of maps
$$
    \Phi_{Y,X} : \hom_\CC(F(Y), X) \to \hom_\DD(Y, G(X))
$$
indexed by the family $\{(Y, X) \in \Ob(\DD) \times \Ob(\CC) : \hom_\CC(F(Y), X) \ne \0\}$
and satisfying the following:
\begin{itemize}
\item[(PA)]
  for every $C \in \Ob(\CC)$, every $D, E \in \Ob(\DD)$,
  every $u \in \hom_\CC(F(D), C)$ and every $f \in \hom_\DD(E, D)$ there is a $v \in \hom_\CC(F(E), F(D))$
  satisfying $\Phi_{D, C}(u) \cdot f = \Phi_{E, C}(u \cdot v)$.
  $$
    \XYMATRIX{
      F(D) \ar[rr]^u                     & & C & & D \ar[rr]^{\Phi_{D, C}(u)}                    & & G(C) \\
      F(E) \ar[u]^v \ar[urr]_{u \cdot v} & &       & & E \ar[u]^f \ar[urr]_{\Phi_{E, C}(u \cdot v)}
    }
  $$
\end{itemize}
(Note that in a pre-adjunction $F$ and $G$ are \emph{not} required to be functors, just maps from the class of objects of one of the two
categories into the class of objects of the other category; also $\Phi$ is just a family of
maps between hom-sets satisfying the requirement above.)

\begin{THM} \cite{masul-preadj} \label{opos.thm.main}
  Let $\CC$ and $\DD$ be categories and let $F : \Ob(\DD) \rightleftarrows \Ob(\CC) : G$ be a pre-adjunction with
  $\Phi_{Y,X} : \hom_\CC(F(Y), X) \to \hom_\DD(Y, G(X))$ as the corresponding
  family of maps between hom-sets. Assume that $\CC$ has the Ramsey property. Then $\DD$ has the Ramsey property.
\end{THM}
\begin{proof}
  (Sketch)
  Take any $D, E \in \Ob(\DD)$ and an integer $k \ge 2$. Since $\CC$ has the Ramsey property, there is a $C \in \Ob(\CC)$
  such that $C \longrightarrow (F(D))^{F(E)}_k$. Let us show that $G(C) \longrightarrow (D)^{E}_k$.
  Take any coloring $\hom_\DD(E, G(C)) = \calX_1 \union \ldots \union \calX_k$ and construct a coloring
  $\hom_\CC(F(E), C) = \calX'_1 \union \ldots \union \calX'_k$ where
  $
    \calX'_i = \{u \in \hom_\CC(F(E), C) : \Phi_{E, C}(u) \in \calX_i\}
  $.
  By the choice of $C$ there is a $u \in \hom_\CC(F(D), C)$ and a $j \in \{1, \ldots, k\}$ such that
  $
    u \cdot \hom_\CC(F(E), F(D)) \subseteq \calX'_j
  $.
  Then it is easy to show that
  $
    \Phi_{D, C}(u) \cdot \hom_\DD(E, D) \subseteq \calX_j
  $.
  Namely, take any $f \in \hom_\DD(E, D)$. Since $F : \Ob(\DD) \rightleftarrows \Ob(\CC) : G$ is a pre-adjunction, there is a
  $v \in \hom_\CC(F(E), F(D))$ such that
  $
    \Phi_{D, C}(u) \cdot f = \Phi_{E, C}(u \cdot v)
  $.
  But then $u \cdot v \in \calX'_j$, so $\Phi_{E, C}(u \cdot v) \in \calX_j$.
  Therefore, $\Phi_{D, C}(u) \cdot f  \in \calX_j$.
\end{proof}

In other words, we have just shown that ``right pre-adjoints'' preserve the Ramsey property.
En passant, let us mention that the Ramsey property is invariant under categorical equivalence, and that right adjoints
preserve the Ramsey property while left adjoints preserve its dual~\cite{masulovic-ramsey}.
(A category $\CC$ has the \emph{dual Ramsey property} if $\CC^\op$ has the Ramsey property.)

An important transfer principle is the Product Ramsey Theorem for Finite Structures of M.~Soki\'c~\cite{sokic2}.
We proved this statement in the categorical context in~\cite{masul-drp-perm} where we used this abstract version
to prove that the class of finite permutations has the dual Ramsey property.

\begin{THM}\label{sokic-prod} \cite{masul-drp-perm}
  Let $\CC_1$ and $\CC_2$ be categories such that $\hom_{\CC_i}(A, B)$ is finite
  for all $A, B \in \Ob(\CC_i)$, $i \in \{1, 2\}$.
  If $\CC_1$ and $\CC_2$ both have the Ramsey property then
  $\CC_1 \times \CC_2$ has the Ramsey property.
  
  Consequently, if $\CC_1, \ldots, \CC_n$ are categories with the Ramsey property then the category
  $\CC_1 \times \ldots \times \CC_n$ has the Ramsey property.
\end{THM}
\begin{proof}
  (Sketch)
  Take any $k \ge 2$ and $\tilde A = (A_1, A_2)$, $\tilde B = (B_1, B_2)$ in $\Ob(\CC_1 \times \CC_2)$
  such that $\hom(\tilde A, \tilde B) \ne \0$. Take $C_1 \in \Ob(\CC_1)$ and
  $C_2 \in \Ob(\CC_2)$ so that $C_1 \longrightarrow (B_1)^{A_1}_{k}$ and
  $C_2 \longrightarrow (B_2)^{A_2}_{k^t}$,
  where $t$ is the cardinality of $\hom_{\CC_1}(A_1, C_1)$. Put $\tilde C = (C_1, C_2)$.
  To show that $\tilde C \longrightarrow (\tilde B)^{\tilde A}_k$ take any coloring
  $
    \chi : \hom_{\CC_1 \times \CC_2}({\tilde A},{\tilde C}) \to \{1, \ldots, k\}
  $.
  This coloring uniquely induces the $k^t$-coloring
  $
    \chi' : \hom_{\CC_2}({A_2},{C_2}) \to \{1, \ldots, k\}^{\hom_{\CC_1}({A_1},{C_1})}
  $.
  By construction, $C_{2} \longrightarrow (B_{2})^{A_{2}}_{k^t}$,
  so there is a $w_2 : B_2 \to C_2$ such that $w_2 \cdot \hom_{\CC_2}({A_2},{B_2})$ is $\chi'$-monochromatic.
  Define $\chi'' : \hom_{\CC_1}({A_1},{C_1}) \to \{1, \ldots, k\}$ by
  $\chi''(e_1) = \chi(e_1, e)$ for some $e \in w_2 \cdot \hom_{\CC_2}({A_2},{B_2})$. (Note that $\chi''$ is well defined because
  $w_2 \cdot \hom_{\CC_2}({A_2},{B_2})$ is $\chi'$-monochromatic.)
  Since $\CC_1$ has the Ramsey property there is a morphism
  $
    w_1 : B_1 \to C_1
  $
  such that $w_1 \cdot \hom_{\CC_1}({A_1},{B_1})$
  is $\chi''$-monochromatic. It is now easy to show that for $\tilde w = (w_1, w_2)$ we have that
  $\tilde w \cdot \hom_{\CC_1 \times \CC_2}({(A_1, A_2)},{(B_1, B_2)})$ is $\chi$-monochromatic.
\end{proof}

Finally, we shall also need a way to transfer the Ramsey property from a category to its subcategory. (For many deep results
obtained in this fashion see~\cite{Nesetril-AllThose}.) In~\cite{masul-drp-perm} we devised
a simple result which enables us to transfer the Ramsey property from a category to its (not necessarily full) subcategory,
as follows.

A \emph{diagram} in a category $\CC$ is a functor $F : \Delta \to \CC$ where the category $\Delta$ is referred to as the
\emph{shape of the diagram}. We shall say that a diagram $F : \Delta \to \CC$ is \emph{consistent in $\CC$} if there exists a $C \in \Ob(\CC)$
and a family of morphisms $(e_\delta : F(\delta) \to C)_{\delta \in \Ob(\Delta)}$ such that for every
morphism $g : \delta \to \gamma$ in $\Delta$ we have $e_\gamma \cdot F(g) = e_\delta$:
$$
  \xymatrix{
     & C & \\
    F(\delta) \ar[ur]^{e_\delta} \ar[rr]_{F(g)} & & F(\gamma) \ar[ul]_{e_\gamma}
  }
$$
We say that $C$ together with the family of morphisms
$(e_\delta)_{\delta \in \Ob(\Delta)}$ forms a \emph{compatible cone in $\CC$ over the diagram~$F$}.

A \emph{binary category} is a finite, acyclic, bipartite digraph with loops 
where all the arrows go from one class of vertices into the other
and the out-degree of all the vertices in the first class is~2 (modulo loops):
$$
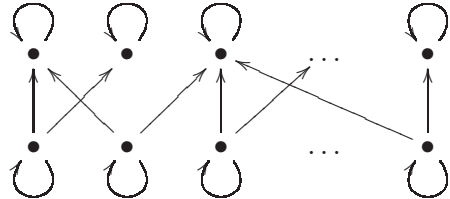
\xymatrix{
  & \\
  \bullet \ar@(ur,ul) & \bullet \ar@(ur,ul) & \bullet \ar@(ur,ul) & \ldots & \bullet \ar@(ur,ul) \\
  \bullet \ar@(dr,dl) \ar[u] \ar[ur] & \bullet \ar@(dr,dl) \ar[ur] \ar[ul] & \bullet \ar@(dr,dl) \ar[u] \ar[ur] & \ldots & \bullet \ar@(dr,dl) \ar[u] \ar[ull] \\
  &
}
$$
A \emph{binary diagram} in a category $\CC$ is a functor $F : \Delta \to \CC$ where $\Delta$ is a binary category,
$F$ takes the bottom row of $\Delta$ onto the same object, and takes the top row of $\Delta$ onto
the same object, Fig.~\ref{nrt.fig.2}.
A subcategory $\DD$ of a category $\CC$ is \emph{closed for binary diagrams} if every binary diagram
$F : \Delta \to \DD$ which is consistent in $\CC$ is also consistent in~$\DD$.

\begin{figure}
  $$
  \xymatrix{
    \bullet & \bullet & \bullet
    & & B & B & B
  \\
    \bullet \ar[u] \ar[ur] & \bullet \ar[ur] \ar[ul] & \bullet \ar[ul] \ar[u]
    & & A \ar[u]^{f_1} \ar[ur]_(0.3){f_2} & A \ar[ur]^(0.3){f_4} \ar[ul]_(0.3){f_3} & A \ar[ul]^(0.3){f_5} \ar[u]_{f_6}
  \\
    & \Delta \ar[rrrr]^F  & & & & \CC  
  }
  $$
  \caption{A binary diagram in $\CC$ (of shape $\Delta$)}
  \label{nrt.fig.2}
\end{figure}
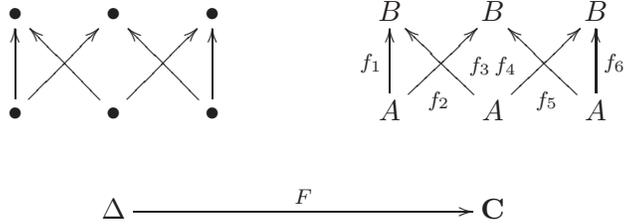

\begin{THM}\label{nrt.thm.1} \cite{masul-drp-perm}
  Let $\CC$ be a category such that every morphism in $\CC$ is monic and
  such that $\hom_\CC(A, B)$ is finite for all $A, B \in \Ob(\CC)$, and let $\DD$ be a
  (not necessarily full) subcategory of~$\CC$. If $\CC$ has the Ramsey property and $\DD$ is closed for binary diagrams,
  then $\DD$ has the Ramsey property.
\end{THM}
\begin{proof}
  (Sketch)
  Take any $k \ge 2$ and $A, B \in \Ob(\DD)$ such that $\hom_\DD(A, B) \ne \0$. Since $\DD$ is a subcategory of $\CC$
  and $\CC$ has the Ramsey property, there is a $C \in \Ob(\CC)$ such that $C \longrightarrow (B)^A_k$.
  
  Let us now construct a binary diagram in $\DD$ as follows. Let $\hom_\CC(B, C) = \{e_1, e_2, \ldots, e_n\}$.
  Intuitively, for each $e_i \in \hom_\CC(B, C)$ we add a copy of $B$ to the diagram, and whenever $e_i \cdot u = e_j \cdot v$
  for some $u, v \in \hom_\DD(A, B)$ we add a copy of $A$ to the diagram together with two arrows: one going into the $i$th copy of $B$
  labelled by $u$ and another one going into the $j$th copy of $B$ labelled by~$v$:
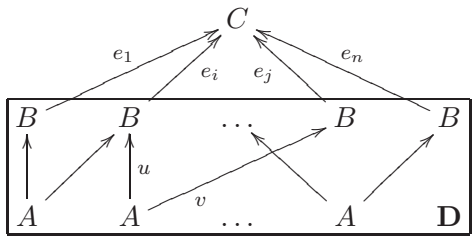
  $$
  \xymatrix{
    & & C
  \\
    B \ar[urr]^{e_1} & B \ar[ur]_(0.6){e_i} & \ldots & B \ar[ul]^(0.6){e_j} & B \ar[ull]_{e_n}
  \\
    A \ar[u] \ar[ur] & A \ar[urr]_(0.3){v} \ar[u]_{u} & \ldots & A \ar[ur] \ar[ul] & \DD
    \save "2,1"."3,5"*[F]\frm{} \restore
  }
  $$
  Note that, by the construction, this diagram is consistent in $\CC$, so,
  by the assumption, it is consistent in~$\DD$ as well. Therefore, there is
  a $D \in \Ob(\DD)$ and morphisms $f_i : B \to D$, $1 \le i \le n$, such that the following diagram in $\DD$ commutes:
  $$
  
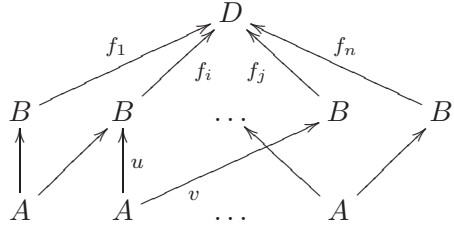
\xymatrix{
    & & D
  \\
    B \ar[urr]^{f_1} & B \ar[ur]_(0.6){f_i} & \ldots & B \ar[ul]^(0.6){f_j} & B \ar[ull]_{f_n}
  \\
    A \ar[u] \ar[ur] & A \ar[urr]_(0.3){v} \ar[u]_{u} & \ldots & A \ar[ur] \ar[ul] & 
  }
  $$
  Let us show that in $\DD$ we have $D \longrightarrow (B)^A_k$. Take any $k$-coloring
  $
    \hom_\DD(A, D) = \calX_1 \union \ldots \union \calX_k
  $,
  and define a $k$-coloring
  $
    \hom_\CC(A, C) = \calX'_1 \union \ldots \union \calX'_k
  $
  as follows. For $j \in \{2, \ldots, k\}$ let
  $
    \calX'_j = \{e_s \cdot u : 1 \le s \le n, u \in \hom_\DD(A, B), f_s \cdot u \in \calX_j \}
  $,
  and then let
  $
    \calX'_1 = \hom_\CC(A, C) \setminus \UNION_{j=2}^k \calX'_j
  $.
  Since $C \longrightarrow (B)^A_k$, there is an $e_\ell \in \hom_\CC(B, C)$ and a $j$ such that
  $e_\ell \cdot \hom_\CC(A, B) \subseteq \calX'_j$. Then it easily follows that
  $f_\ell \cdot \hom_\DD(A, B) \subseteq \calX_j$.
\end{proof}

\section{The Ne\v set\v ril-R\"odl Theorem}
\label{nrt.sec.the-proof}

We are now ready to present a new proof of the Ne\v set\v ril-R\"odl Theorem.

Let us start by showing that for every $b \ge 2$ the category $\catH(b)$ has the Ramsey property.
The proof that we present here is an instance of a more general phenomenon which we addressed in more detail in~\cite{masul-preadj}
and where the main idea of the proof comes from.

\begin{THM} \cite{AH,Nesetril-Rodl-1983}\label{nrt.thm.hgr}
  For every $b \ge 2$ the category $\catH(b)$ has the Ramsey property.
\end{THM}
\begin{proof}
  Fix a $b \ge 2$. In order to prove the theorem it suffices to show that there is a pre-adjunction
  $$
    F : \Ob(\catH(b)) \rightleftarrows \Ob(\catGR(\{0\}, X)) : G,
  $$
  where $X$ is a countably infinite set of variables disjoint from $\{0\}$.
  The result then follows from Theorem~\ref{opos.thm.main} and the fact that the category
  $\catGR(\{0\}, X)$ has the Ramsey property (Example~\ref{opos.ex.GR}).

  Let $\calA = (A, E, \Boxed<)$ be a finite linearly ordered $b$-uniform hypergraph.
  A \emph{downset} in $\calA$ is either
          \sout{a singleton $\{a\}$ where $a \in A$, or a subset $D$ of $A$ such that $|D| \ge 2$ and $D \subseteq e$ for some $e \in E$.}
          \textcolor{blue}{an edge of $\calA$, or a nonempty subset of $A$ of size strictly less than~$b$.}

  For an $\calA \in \Ob(\catH(b))$ let
  $F(\calA) = \mathstrut$the number of distinct nonempty downsets in $\calA$.
  On the other hand, for a positive integer $n$ let $G(n) = \big(\calP(\{1, \ldots, n\}), \calE_n, \Boxed\clex\big)$
  where
  $$
    \calE_n = \big\{ \{Y_1, \ldots, Y_b\} : Y_1, \ldots, Y_b \in \calP(\{1, \ldots, n\})
    \text{ and } Y_1 \sec \ldots \sec Y_b \ne \0 \big\}.
  $$
  For a finite linearly ordered $b$-uniform hypergraph $\calA$ and a positive integer $n$ define
  $$
    \Phi_{\calA, n} : \hom_{\catGR(\{0\}, X)}(F(\calA), n) \to \hom_{\catH(b)}(\calA, G(n))
  $$
  as follows. Let $\calA = (\{1, 2, \ldots, k\}, E, \Boxed<)$ where $<$ is the usual
  ordering of the integers. Let $D_1$, \ldots, $D_m$ be all the nonempty downsets in $\calA$ and let
  $D_1 \alex D_2 \alex \ldots \alex D_m$.
  For $u \in \hom_{\catGR(\{0\}, X)}(F(\calA), n) = W^n_m(\{0\})$, let $X_i = u^{-1}(x_i)$, $1 \le i \le m$, and let
  $$
    a_i = \UNION\{X_\alpha : i \in D_\alpha\},
  $$
  $1 \le i \le k$. Put $\Phi_{\calA, n}(u) = \hat u$
  where $\hat u : \calA \to G(n) : i \mapsto a_i$.

  To show that the definition of $\Phi$ is corect we have to show that for every $u \in W^n_m(\{0\})$
  the mapping $\hat u$ is an embedding $\calA \hookrightarrow G(n)$. Take any $i_1, \ldots, i_b \in A$
  such that $\{i_1, \ldots, i_b\} \in E$ and let $\{i_1, \ldots, i_b\} = D_\alpha$. Then
  $a_{i_1} \sec \ldots \sec a_{i_b} \supseteq X_\alpha \ne \0$, so $\{ a_{i_1}, \ldots, a_{i_b} \} \in \calE_n$.
  
  On the other hand, assume that $\{ a_{i_1}, \ldots, a_{i_b} \} \in \calE_n$. Then
  $a_{i_1} \sec \ldots \sec a_{i_b} \ne \0$. Since each $a_j$ is a union of some $X_\alpha$'s and all the
  $X_\alpha$'s are pairwise disjoint, it follows that $a_{i_1} \sec \ldots \sec a_{i_b}$ is also a union of some
  $X_\alpha$'s. Therefore, there is a $\beta$ such that $a_{i_1} \sec \ldots \sec a_{i_b} \supseteq X_\beta$.
  Then $i_1, \ldots, i_b \in D_\beta$. Since $i_1, \ldots, i_b$ are pairwise distinct and $|D_\beta| \le b$
  by the definition of the downset, it follows that $D_\beta = \{i_1, \ldots, i_b\}$. Since the only $b$-elements
  downsets of $\calA$ are hyperedges of $\calA$, it follows that $\{i_1, \ldots, i_b\} \in E$.
  
  Finally, let us show that $i < j$ implies $a_i \clex a_j$. Since both $\{i\}$ and $\{j\}$ are
  downsets in $\calA$ there exist $\eta$ and $\xi$ such that $D_\eta = \{i\}$ and $D_\xi = \{j\}$.
  Then $X_\eta \subseteq a_i$ and $X_\eta \sec a_j = \0$, while $X_\xi \subseteq a_j$ and $X_\xi \sec a_i = \0$,
  whence follows that $a_i$ and $a_j$ are incomparable as sets. Note also that $\{i\} \alex \{j\} \alex D$ for
  every downset $D$ in $\calA$ such that $D \ni j$, so
  $$
    \min(a_i \setminus a_j) \le \min X_\eta < \min(a_j \setminus a_i).
  $$
  Therefore, $a_i \clex a_j$.

  So, the definition of $\Phi$ is correct. We still have to show that this family of maps
  satisfies the requirement~(PA).

  Let $\calB = (\{1, 2, \ldots, \ell\}, F, \Boxed<)$ be a finite linearly ordered $b$-uniform hypergraph
  that embeds into $\calA$.
  Let $D'_1$, \ldots, $D'_d$ be all the nonempty downsets in $\calB$ and let $D'_1 \alex D'_2 \alex \ldots \alex D'_d$.
  Take any embedding $f : \calB \hookrightarrow \calA$ and let us show that there is a word
  $h = h_1 h_2 \ldots h_m \in W^m_d(\{0\})$ such that $\Phi_{\calA, n}(u) \circ f = \Phi_{\calB, n}(u \cdot h)$.

  Define $h = h_1 h_2 \ldots h_m \in W^m_d(\{0\})$ as follows:
  $$
    h_i = \begin{cases}
      x_j, & f^{-1}(D_i) = D'_j\\
      0,   & \text{otherwise.}
    \end{cases}
  $$
  Let us first show that $h$ is indeed a $d$-parameter word. It is easy to see that
  every downset in $\calB$ is an inverse image of a downset in $\calA$
  so each of the variables $x_1, \ldots, x_d$ appears at least once in $h$.
  
  Let us show that $\min(h^{-1}(x_\alpha)) < \min(h^{-1}(x_\beta))$ whenever $1 \le \alpha < \beta \le d$.
  Take $\alpha$, $\beta$ such that $1 \le \alpha < \beta \le d$ and let $\min(h^{-1}(x_\beta)) = q$. Since $h_q = x_\beta$
  we know that $f^{-1}(D_q) = D'_\beta$. Take $p$ so that $D'_\alpha = f^{-1}(D_p)$. Then $h_p = x_\alpha$.
  If $p < q$ then $\min(h^{-1}(x_\alpha)) \le p < q = \min(h^{-1}(x_\beta))$ and we are done. Assume, therefore, that $p > q$.
  So, we have that $D_p \galex D_q$ (because $p > q$) and $f^{-1}(D_p) \alex f^{-1}(D_q)$ (because
  $\alpha < \beta$ whence $f^{-1}(D_p) = D'_\alpha \alex D'_\beta = f^{-1}(D_q)$).
  Let us show that in this case there exists a $D_r$ such that $D_r \alex D_q$ and $f^{-1}(D_r) = f^{-1}(D_p)$.
  Put $D_r = f(f^{-1}(D_p))$ so that $f^{-1}(D_r) = f^{-1}(D_p)$. We know that
  $f^{-1}(D_p) \alex f^{-1}(D_q)$, so $f(f^{-1}(D_p)) \alex f(f^{-1}(D_q))$.
  Since $f(f^{-1}(D_q)) \subseteq D_q$ it follows that $D_r = f(f^{-1}(D_p)) \alex f(f^{-1}(D_q)) \alex D_q$
  (as $\alex$ extends $\subseteq$). Now, $D_r \alex D_q$ implies that $r < q$,
  while $f^{-1}(D_r) = f^{-1}(D_p) = D'_\alpha$ means that $h_r = x_\alpha$.
  Therefore, $\min(h^{-1}(x_\alpha)) \le r < q = \min(h^{-1}(x_\beta))$, which completes the proof that $h$ is a $d$-parameter word.

  Let $X'_i = (u \cdot h)^{-1}(x_i)$, $1 \le i \le d$, and $a'_j = \UNION \{X'_\beta : j \in D'_\beta\}$, $1 \le j \le \ell$.
  The following is a strightforward but useful observation:
  if $h^{-1}(x_\beta) = \{\alpha_1, \ldots, \alpha_s\}$ then
  $D'_\beta = f^{-1}(D_{\alpha_1}) = \ldots = f^{-1}(D_{\alpha_s})$
  and $X'_\beta = (u \cdot h)^{-1}(x_\beta) = u^{-1}(x_{\alpha_1}) \union \ldots \union u^{-1}(x_{\alpha_s})
  = X_{\alpha_1} \union \ldots \union X_{\alpha_s}$.
  
  In order to complete the proof it suffices to show that $a'_j = a_{f(j)}$ for all $1 \le j \le \ell$.
  
  $(\subseteq)$:
  Take any $X'_\beta \subseteq a'_j$. Then $j \in D'_\beta$. Let $h^{-1}(x_\beta) = \{\alpha_1, \ldots, \alpha_s\}$. Then
  $D'_\beta = f^{-1}(D_{\alpha_1}) = \ldots = f^{-1}(D_{\alpha_s})$, whence $j \in f^{-1}(D_{\alpha_i})$ for all
  $1 \le i \le s$. Consequently, $f(j) \in D_{\alpha_i}$ for all $1 \le i \le s$, so
  $X_{\alpha_i} \subseteq a_{f(j)}$ for all $1 \le i \le s$. Finally, $X'_\beta = X_{\alpha_1} \union \ldots \union X_{\alpha_s}
  \subseteq a_{f(j)}$.
  
  $(\supseteq)$:
  Take any $X_\alpha \subseteq a_{f(j)}$. Then $f(j) \in D_\alpha$ whence $j \in f^{-1}(D_\alpha) = D'_\beta$
  so $X'_\beta \subseteq a'_j$. By the definition of $h$ we have that $h_\alpha = x_\beta$ whence $X_\alpha \subseteq X'_\beta$.
  Therefore, $X_\alpha \subseteq a'_j$.
\end{proof}

A finite linearly ordered $\Theta$-structure $\calA = (A, \Theta^\calA, \Boxed{<^\calA})$ is \emph{absolutely ordered}
if the following holds for every $R \in \Theta$:
$$
  \text{if } (a_1, a_2, \ldots, a_n) \in R^\calA \text{ then } a_1 \mathrel{<^\calA} a_2 \mathrel{<^\calA} \ldots \mathrel{<^\calA} a_n.
$$
Let $\catAREL(\Theta, \Boxed<)$ denote the class of all finite
absolutely ordered $\Theta$-structures.

\begin{LEM}\label{nrt.lem.AREL-FIN}
  The category $\catAREL(\Theta, \Boxed{<})$ has the Ramsey property
  for every finite relational language~$\Theta$.
\end{LEM}
\begin{proof}
  Assume, first, that $\Theta = \{R\}$ where $R$ is an $r$-ary relational symbol. Then it is easy to see that
  the categories $\catAREL(\{R\}, \Boxed{<})$ and $\catH(r)$ are isomorphic,
  so $\catAREL(\{R\}, \Boxed{<})$ also has the Ramsey property.

  Assume, now, that $\Theta = \{ R_1, R_2, \ldots, R_n \}$ is a finite relational language.
  Let $\CC_i$ denote the category $\catAREL(\{R_i\}, \Boxed<)$, $1 \le i \le n$.
  For an object $\calA = (A, R_1^\calA, \ldots, R_n^\calA, \Boxed{<}) \in \Ob(\catAREL(\Theta, \Boxed{<}))$ let
  $\calA^{(i)} = (A, R_i^\calA, \Boxed{<}) \in \Ob(\CC_i)$.
  As we have just seen each $\CC_i$ has the Ramsey property, so the product category
  $\CC_1 \times \ldots \times \CC_n$ has the Ramsey property
  by Theorem~\ref{sokic-prod}. Let $\DD$ be the following subcategory of $\CC_1 \times \ldots \times \CC_n$:
  \begin{itemize}
  \item
    every $\calA = (A, R_1^\calA, \ldots, R_n^\calA, \Boxed{<}) \in \Ob(\catAREL(\Theta, \Boxed{<}))$ gives rise to an object
    $\overline \calA = (\calA^{(1)}, \ldots, \calA^{(n)})$ of $\DD$, and these are the only objects in~$\DD$;
  \item
    every morphism $f : \calA \to \calB$ in $\catAREL(\Theta, \Boxed{<})$ gives rise to a morphism
    $\overline f = (f, \ldots, f) : \overline \calA \to \overline \calB$ in $\DD$, and these are the only morphisms in~$\DD$.
  \end{itemize}
  Clearly, the categories $\DD$ and $\catAREL(\Theta, \Boxed{<})$ are isomorphic,
  so in order to complete the proof of the lemma it suffices to show that $\DD$ has the Ramsey property.
  
  As $\DD$ is a subcategory of $\CC_1 \times \ldots \times \CC_n$ and the latter one has the Ramsey property,
  following Theorem~\ref{nrt.thm.1} it suffices to show that $\DD$ is closed for binary diagrams.
  Let $F : \Delta \to \DD$ be a binary diagram which is consistent in $\CC_1 \times \ldots \times \CC_n$ and let
  $(\calC_1, \ldots, \calC_n)$ together with the morphisms $e_1, \ldots, e_k$ be a compatible cone
  in $\CC_1 \times \ldots \times \CC_n$ over~$F$:
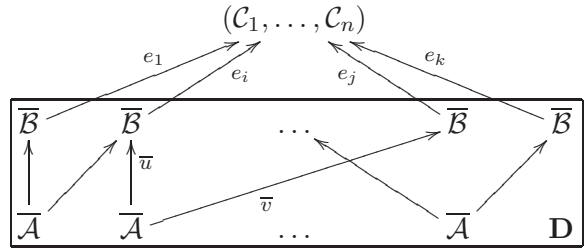
  $$
  \xymatrix{
    & & (\calC_1, \ldots, \calC_n)
  \\
    \overline \calB \ar[urr]^{e_1} & \overline \calB \ar[ur]_(0.6){e_i} & \ldots & \overline \calB \ar[ul]^(0.6){e_j} & \overline \calB \ar[ull]_{e_k}
  \\
    \overline \calA \ar[u] \ar[ur] & \overline \calA \ar[urr]_(0.4){\overline v} \ar[u]_(0.65){\overline u} & \ldots & \overline \calA \ar[ur] \ar[ul] & \DD
    \save "2,1"."3,5"*[F]\frm{} \restore
  }
  $$
  Recall that $\calC_i = (C_i, R_i^{\calC_i}, \Boxed{<})$ and that each $e_i$ is a tuple $e_i = (e_i^1, \ldots, e_i^n)$ where
  $e_i^s : \calB^{(s)} \hookrightarrow \calC_s$. Let
  $
    D = C_1 \times \ldots \times C_n
  $
  and let $\lex$ denote the lexicographic order on $D$ induced by the linear orders $(C_i, \Boxed<)$, $1 \le i \le n$.
  Let $\calD = (D, R_1^\calD, \ldots, R_n^\calD, \Boxed{\lex})$ where
  $$
    R_s^\calD = \{(d_1, d_2, \ldots, d_{r_s}) \in D : (d_1^s, d_2^s, \ldots, d_{r_s}^s) \in R_s^{\calC_s} \text{ and } d_1^1 < \ldots < d_{r_s}^1 \}.
  $$
  Here, $r_s$ is the arity of $R_s$ and $d_j = (d_j^1, \ldots, d_j^n)$.
  To see that $\calD \in \Ob(\catAREL(\Theta))$ it suffices to note that $(d_1, d_2, \ldots, d_{r_s}) \in R_s^\calD$ implies
  $d_1^1 < \ldots < d_{r_s}^1$ whence $d_1 \lex d_2 \lex \ldots \lex d_{r_s}$. Consequently,
  $\overline \calD \in \Ob(\DD)$.
  
  For each morphism $e_i = (e_i^1, \ldots, e_i^n)$ let $\phi_i : B \to D$ be the following mapping:
  $$
    \phi_i(b) = (e_i^1(b), \ldots, e_i^n(b)) \in D.
  $$
  Let us show that $\phi_i : \calB \to \calD$ is an embedding for each~$i$.
  
  Assume, first, that $b < b'$ in $\calB$. Then $e_i^1(b) < e_i^1(b')$ in $\calC_1$ so
  $$
    \phi_i(b) = (e_i^1(b), \ldots, e_i^n(b)) \lex (e_i^1(b'), \ldots, e_i^n(b')) = \phi_i(b').
  $$
  
  Assume, now, that $(b_1, b_2, \ldots, b_{r_s}) \in R_s^\calB$. Then $b_1 < b_2 < \ldots < b_{r_s}$ in $\calB$,
  whence, as we have just seen, $e_i^1(b_1) < e_i^1(b_2) < \ldots < e_i^1(b_{r_s})$. Moreover,
  since $e_i^s : \calB^{(s)} \hookrightarrow \calC_s$ we have that
  $(e_i^s(b_1), e_i^s(b_2), \ldots, e_i^s(b_{r_s})) \in R_s^{\calC_s}$. Therefore,
  $(\phi_i(b_1), \phi_i(b_2), \ldots, \phi_i(b_{r_s})) \in R_s^\calD$.
  
  Conversely, assume that $(\phi_i(b_1), \phi_i(b_2), \ldots, \phi_i(b_{r_s})) \in R_s^\calD$.
  Then\break
  $(e_i^s(b_1), e_i^s(b_2), \ldots, e_i^s(b_{r_s})) \in R_s^{\calC_s}$.
  Since $e^s_i$ is an embedding we immediately conclude that $(b_1, b_2, \ldots, b_{r_s}) \in R_s^\calB$.
  
  Therefore, $\phi_i : \calB \to \calD$ is an embedding for each~$i$, whence follows that $\overline \phi_i : \overline \calB \to
  \overline \calD$ is a morphism in $\DD$ for each~$i$. To complete the proof we still have to show that
  $\overline \phi_i \circ \overline u = \overline \phi_j \circ \overline v$ whenever
  $e_i \circ \overline u = e_j \circ \overline v$. Assume that $e_i \circ \overline u = e_j \circ \overline v$, i.e.
  $$
    (e_i^1 \circ u, e_i^2 \circ u, \ldots, e_i^n \circ u) = (e_j^1 \circ v, e_j^2 \circ v, \ldots, e_j^n \circ v).
  $$
  Then
  \begin{align*}
    \phi_i \circ u &= (e_i^1 \circ u, e_i^2 \circ u, \ldots, e_i^n \circ u) \\
                   &= (e_j^1 \circ v, e_j^2 \circ v, \ldots, e_j^n \circ v) \\
                   &= \phi_j \circ v
  \end{align*}
  whence easily follows that $\overline \phi_i \circ \overline u = \overline \phi_j \circ \overline v$. This concludes the proof.
\end{proof}

\begin{THM}[Ne\v set\v ril-R\"odl] \cite{AH,Nesetril-Rodl-1983}
  $\catREL(\Theta, \Boxed{<})$ has the Ramsey property for every relational language~$\Theta$.
\end{THM}
\begin{proof}
  Let us first show that $\catAREL(\Theta, \Boxed{<})$ has the Ramsey property for every relational language~$\Theta$.
  Fix an arbitrary relational language $\Theta$ such that $\Boxed< \notin \Theta$ and take any $k \ge 2$ and
  $\calA, \calB \in \Ob(\catAREL(\Theta, \Boxed{<}))$ such that $\calA \hookrightarrow \calB$.
  Since $\calB$ is a finite absolutely ordered $\Theta$-structure we have that
  $R^\calB = \0$ for every $R \in \Theta$ such that $\arity(R) > |B|$. Moreover, on a finite set there are only finitely many
  relations whose arities do not exceed $|B|$. Therefore, there exists a finite $\Sigma \subseteq \Theta$ such that
  for every $R \in \Theta \setminus \Sigma$ we have $R^\calB = \0$ or $R^\calB = S^\calB$ for some $S \in \Sigma$.
  Since $\calA \hookrightarrow \calB$ we have the following:
  if $R^\calB = \0$ for some $R \in \Theta \setminus \Sigma$ then $R^\calA = \0$, and
  if $R^\calB = S^\calB$ for some $R \in \Theta \setminus \Sigma$ and $S \in \Sigma$ then
  $R^\calA = S^\calA$.
  
  The category $\catAREL(\Sigma, \Boxed{<})$ has the Ramsey property because $\Sigma$ is finite (Lemma~\ref{nrt.lem.AREL-FIN}),
  so there is a $\calC = (C, \Sigma^\calC, \Boxed{<^\calC}) \in \Ob(\catAREL(\Sigma, \Boxed{<}))$ such that
  $$
    \calC \longrightarrow (\reduct \calB {\Sigma \union \{<\}})^{\reduct \calA {\Sigma \union \{<\}}}_k.
  $$
  Define $\calC^* = (C, \Theta^{\calC^*}, \Boxed{<^{\calC^*}})
  \in \Ob(\catAREL(\Theta, \Boxed{<}))$ as follows:
  \begin{itemize}
  \item
    \Boxed{<^{\calC^*}} = \Boxed{<^{\calC}};
  \item
    if $S \in \Sigma$ let $S^{\calC^*} = S^{\calC}$;
  \item
    if $R \in \Theta \setminus \Sigma$ and $R^\calB = \0$ let $R^{\calC^*} = \0$;
  \item
    if $R \in \Theta \setminus \Sigma$ and $R^\calB = S^\calB$ for some $S \in \Sigma$, let $R^{\calC^*} = S^{\calC^*}$.
  \end{itemize}
  Clearly, $\calC^*$ is a finite absolutely ordered $\Theta$-structure and $\calC^* \longrightarrow (\calB)^\calA_k$.

  Finally, let us show that $\catREL(\Theta, \Boxed{<})$ has the Ramsey property for every relational language~$\Theta$.
  We start by recalling some basic facts about total quasiorders.
  
  A total quasiorder is a reflexive and transitive binary relation such that each pair of elements
  of the underlying set is comparable. Each total quasiorder $\sigma$ on a set $I$ induces an equivalence relation $\equiv_\sigma$
  on $I$ and a linear order $\sqsubset_\sigma$ on $I / \Boxed{\equiv_\sigma}$ in a natural way: $i \mathrel{\equiv_\sigma} j$ if
  $(i, j) \in \sigma$ and $(j, i) \in \sigma$, and $(i / \Boxed{\equiv_\sigma}) \mathrel{\sqsubset_\sigma} (j / \Boxed{\equiv_\sigma})$
  if $(i, j) \in \sigma$ and $(j, i) \notin \sigma$.

  Let $(A, \Boxed<)$ be a linearly ordered set, let $r$ be a positive integer, let $I = \{1, \ldots, r\}$ and
  let $\overline a = (a_1, \ldots, a_r) \in A^r$. Then
  $$
    \tp(\overline a) = \{(i, j) : a_i \le a_j \}
  $$
  is a total quasiorder on $I$ which we refer to as the \emph{type} of~$\overline a$.
  Assume that $\sigma = \tp(\overline a)$. Let $s = |I / \Boxed{\equiv_\sigma}|$ and let $i_1$, \ldots, $i_s$
  be the representatives of the classes of $\equiv_\sigma$ enumerated so that
  $(i_1 / \Boxed{\equiv_\sigma}) \mathrel{\sqsubset_\sigma} \ldots \mathrel{\sqsubset_\sigma} (i_s / \Boxed{\equiv_\sigma})$.
  Then
  $$
    \mat(\overline a) = (a_{i_1}, \ldots, a_{i_s})
  $$
  is the \emph{matrix of $\overline a$}. Note that $a_{i_1} < \ldots < a_{i_s}$.
  
  Conversely, given a matrix and a total
  quasiorder we can always reconstruct the original tuple as follows. For a total quasiorder $\sigma$ on $I$ such that
  $|I / \Boxed{\equiv_\sigma}| = s$ and an $s$-tuple $\overline b = (b_1, \ldots, b_s) \in A^s$ such that $b_1 < \ldots < b_s$
  define an $r$-tuple
  $$
    \tup(\sigma, \overline b) = (a_1, \ldots, a_r) \in A^r
  $$
  as follows. Let $i_1$, \ldots, $i_s$ be the representatives of the classes of $\equiv_\sigma$ enumerated so that
  $(i_1 / \Boxed{\equiv_\sigma}) \mathrel{\sqsubset_\sigma} \ldots \mathrel{\sqsubset_\sigma} (i_s / \Boxed{\equiv_\sigma})$.
  Then put
  $$
    a_\eta = b_\xi \text{ if and only if } \eta \mathrel{\equiv_\sigma} i_\xi.
  $$
  (In other words, we put $b_1$ on all the entries in $i_1 / \Boxed{\equiv_\sigma}$, we put
  $b_2$ on all the entries in $i_2 / \Boxed{\equiv_\sigma}$, and so on.) Then it is a matter of routine to check that
  \begin{equation}\label{nrt.eq.tup-mat}
    \begin{aligned}
      \tp(\tup(\sigma, \overline b)) = \sigma, \text{ } \mat(\tup(\sigma, \overline b))    &= \overline b, \text{ and}\\
      \tup(\tp(\overline a), \mat(\overline a)) &= \overline a,\\
    \end{aligned}
  \end{equation}
  Now, for a relational language $\Theta$ such that $\Boxed< \notin \Theta$ let
  $$
    X_\Theta = \{ (R, \sigma) : R \in \Theta \text{ and } \sigma \text{ is a total quasiorder on } \{1, 2, \ldots, \arity(R)\} \}
  $$
  be a relational language where $\arity(R, \sigma) = |I / \Boxed{\equiv_\sigma}|$. For 
  $\calA = (A, \Theta^\calA, \Boxed{<^\calA}) \in \Ob(\catREL(\Theta, \Boxed{<}))$
  define a $\calA^\dagger = (A, X_\Theta^{\calA^\dagger}, \Boxed{<^{\calA^\dagger}})$ as follows:
  \begin{align*}
    \Boxed{<^{\calA^\dagger}} &= \Boxed{<^{\calA}},\\
    (R, \sigma)^{\calA^\dagger} &= \{\mat(\overline a): \overline a \in R^\calA \text{ and } \tp(\overline a) = \sigma \}.
  \end{align*}
  Clearly, $\calA^\dagger \in \Ob(\catAREL(X_\Theta, \Boxed<))$. On the other hand, take any
  $\calB = (B, X_\Theta^\calB, \Boxed{<^\calB}) \in \Ob(\catAREL(X_\Theta, \Boxed<))$
  and define $\calB^* = (B, \Theta^{\calB^*}, \Boxed{<^{\calB^*}}) \in \Ob(\catREL(\Theta, \Boxed<))$ as follows:
  \begin{align*}
    \Boxed{<^{\calB^*}} &= \Boxed{<^{\calB}},\\
    R^{\calB^*} &= \{\tup(\sigma, \overline a) : \sigma \text{ is a total quasiorder on } \{1, 2, \ldots, \arity(R)\}\\
                &\qquad\qquad\qquad\qquad\text{and } \overline a \in (R, \sigma)^\calB\}.
  \end{align*}
  Because of \eqref{nrt.eq.tup-mat} we have that $(\calA^\dagger)^* = \calA$ and $(\calB^*)^\dagger = \calB$ for all
  $\calA \in \Ob(\catREL(\Theta, \Boxed{<}))$ and all $\calB \in \Ob(\catAREL(X_\Theta, \Boxed{<}))$.
  Therefore, the functor
  $$
    F : \catREL(\Theta, \Boxed{<}) \to \catAREL(X_\Theta, \Boxed{<}) : \calA \mapsto \calA^\dagger : f \mapsto f
  $$
  is an isomorphism between the categories $\catREL(\Theta, \Boxed{<})$ and $\catAREL(X_\Theta, \Boxed{<})$, its inverse being
  $$
    G : \catAREL(X_\Theta, \Boxed{<}) \to \catREL(\Theta, \Boxed{<}) : \calB \mapsto \calB^* : f \mapsto f.
  $$
  Since $\catAREL(\Theta, \Boxed{<})$ has the Ramsey property and $\catREL(\Theta, \Boxed{<})$ is isomorphic to
  $\catAREL(X_\Theta, \Boxed{<})$, it follows immediately that $\catREL(\Theta, \Boxed{<})$ has the Ramsey property.
\end{proof}

\section{Acknowledgements}

The author would like to thank Miodrag Soki\'c and Jan Hubi\v cka for fruitful discussions concerning earlier versions
which improved the clarity, correctness and consistency of the paper.

\textcolor{blue}{The author would also like to express his gratitude to Dana Barto\v sov\'a and Valentino Vito
for bringing to my attention a mistake in the earlier version of the paper, for suggesting a new notion
of downsets in the proof of Theorem~4.1 and verifying that this new notion makes the proof of Theorem~4.1
correct.}

The author gratefully acknowledges the support of the Grant No.\ 174019 of the Ministry of Education, Science and Technological Development of the Republic of Serbia.

\end{document}